\documentclass[10pt, leqno]{amsart}

\usepackage{amsmath,amssymb,amsthm, epsfig}
\usepackage{hyperref}
\usepackage{amsmath}
\usepackage{amssymb}
\usepackage{color}
\usepackage{ulem}
\usepackage{epsfig}
\usepackage[mathscr]{eucal}

\usepackage[latin1]{inputenc}

\newtheorem{theorem}{Theorem}
\newtheorem{definition}{Definition}

\newtheorem{lemma}{Lemma}
\newtheorem{proposition}{Proposition}
\newtheorem{corollary}{Corollary}

\date{}
\numberwithin{equation}{section} \numberwithin{theorem}{section}
\numberwithin{lemma}{section} \numberwithin{corollary}{section}
\numberwithin{remark}{section} \numberwithin{proposition}{section}
\numberwithin{definition}{section}

\newcommand{\loc}{\operatorname{loc}}

\def \supp {\mathrm{supp } }

\def \loc {\mathrm{loc}}

\def \suchthat {\ \big | \ }

\newcommand{\pe}{E_{\varepsilon}}

\newcommand{\intav}[1]{\mathchoice {\mathop{\vrule width 6pt height 3 pt depth  -2.5pt
\kern -8pt \intop}\nolimits_{\kern -6pt#1}} {\mathop{\vrule width
5pt height 3  pt depth -2.6pt \kern -6pt \intop}\nolimits_{#1}}
{\mathop{\vrule width 5pt height 3 pt depth -2.6pt \kern -6pt
\intop}\nolimits_{#1}} {\mathop{\vrule width 5pt height 3 pt depth
-2.6pt \kern -6pt \intop}\nolimits_{#1}}}

\begin{document}

\title[Free boundary problem]{POROSITY OF THE FREE BOUNDARY FOR $p$-PARABOLIC TYPE EQUATIONS IN NON-DIVERGENCE FORM}

\author[G. C. Ricarte]{Gleydson C. Ricarte}
\address{Universidade Federal do Cear\'{a} - UFC, Department of Mathematics, Fortaleza - CE, Brazil - 60455-760.}
\email{ricarte@mat.ufc.br}

\begin{abstract}
In this article we establish the exact growth of the solution to the singular quasilinear $p$-parabolic free boundary problem in non-divergence form near the free boundary from which follows its porosity.

\bigskip


\bigskip

\noindent \textbf{AMS Subject Classifications MSC 2010:} 35K55, 35D40, 35B65, 35R35.

\end{abstract}

\maketitle

\section{Introduction}

In this paper we prove an uniform gradient estimate for solutions to  a one phase free boundary problem involving  singular $p$-parabolic type equations in non-divergence form. This estimate then allows us to obtain porosity of the free boundary.  More precisely, let $\Omega$ be an open bounded domain of $\mathbb{R}^n$, $n \ge 2$, $T>0$. We consider the problem: Find $0 \le u \in  C((0,T) \times \Omega)$ such that
 \begin{equation} \label{P1}
 \left\{
\begin{array}{rclcl}
\Delta^N_p u-u_t& = & f \cdot \chi_{\{u>0\}} & \mbox{in} & \Omega_T \\
u & = & \varphi & \mbox{on} & \partial_{p} \Omega_T,
\end{array}
\right.
\end{equation}
where $p \in (1,+\infty)$, $\Delta^N_p$ is the normalized $p$-Laplacian defined by
\begin{eqnarray*}
	\Delta^N_p u &\colon=& \frac{1}{p} |\nabla u|^{2-p} \textrm{div} \left( |\nabla u|^{p-2} \nabla u\right)\\
	&=&\frac{1}{p} \Delta u + \frac{p-2}{p} \Delta^N_{\infty} u,
\end{eqnarray*}
 and $\Delta^N_{\infty} u \colon = \left\langle D^2 u \frac{\nabla u}{|\nabla u|}, \frac{\nabla u}{|\nabla u|} \right \rangle$ is a normalized infinity Laplacian, $f : \Omega_T \rightarrow \mathbb{R}$ is a function non-increasing in $t$ and satisfying for two positive constants $c_0$ and $c_1$ 
\begin{equation} \label{E1}
0 < c_0 \le f(x,t) \le c_1 < \infty  \quad \textrm{in} \quad \Omega_T,
\end{equation}  
 $\varphi(x,0)=0$ in $\Omega$ and $\varphi$ is non-decreasing in $t$. We are interested in regularity theory and geometric properties of $\partial \{u>0\}$.  For this, we use the regularity theory of normalized $p$-Laplacian parabolic equations which was recently developed in \cite{AM,Does,Silv}. When we consider free boundary problems, optimal regularity results and sharp non-degeneracy are crucial for further analysis of the set $\partial \{u>0\}$. In this direction, just recently we have the work \cite{RMR}. 

Equation \eqref{P1} is the degenerate fully nonlinear version of the parabolic problem studied by H. Shahgholian \& G. Ricarte, J.M. Urbano and R. Teymurazyan  in \cite{Sha} and \cite{RMR} respectively. They proved optimal regularity and non-degeneracy estimates for the solution. As a consequence, they were able to obtain finite speed of propagation of the set $\partial \{u>0\}$ and also Hausdorff measure estimate of the free boundary $\partial \{u>0\}$ with respect to the parabolic metric.  

In analogy with \cite{RMR}, what we want  to prove is that, at each time level, the n-dimensional Lebesgue measure of the free boundary is zero because it is porous. We prove this by obtaining a non-degeneracy result and by controlling the growth rate of the solution near the free boundary. As in \cite{RMR}, the solution to \eqref{P1} is derived from an approximating family of functions, which are solutions to some Dirichlet problems. More precisely,   we consider the following singular perturbation problem 
\begin{equation}\label{Equation Pe} \tag{$\pe$}
\left\{
\begin{array}{rclcl}
 \mathcal{L} u^{\varepsilon} - \partial_t u_{\varepsilon} & = & \zeta_{\varepsilon}(u_{\varepsilon}) + f_{\varepsilon} & \mbox{in} & \Omega_T \\
u_{\varepsilon} & = & \varphi & \mbox{on} & \partial_{p} \Omega_T,
\end{array}
\right.
\end{equation}
where $\mathcal{L} v \colon= \frac{1}{p}\Delta v + \frac{p-2}{p} \Delta^{N}_{\infty} v$. The singularly perturbed potential $\zeta_{\varepsilon}(\cdot)$ is a suitable approximation of a multiple of the Dirac mass $\delta_0$. Our objective is to study the limit problem as $\varepsilon \to 0$ and analyze the free boundary in the context of geometric measure theory.  In the next section we give more details of our results and hypothesis.

\bigskip

\noindent{\bf Acknowledgments.}  GCR thanks the Analysis research group of UFC for fostering a pleasant and productive scientific atmosphere. The author research has been partially funded by FUNCAP-Brazil.

\section{Main results}

Let us now describe in more details our results and hypothesis. 

\subsection{Function spaces and notations}
Let us fix the notation of geometric quantities that we are going to use in this work. Given a bounded domain $\Omega \subset \mathbb{R}^n$, with a smooth boundary $\partial \Omega$, we define, for $T>0$, $\Omega_T = \Omega \times (0,T]$, its lateral boundary $\Sigma = \partial \Omega \times (0,T)$ and its parabolic boundary $\partial_p \Omega_T = \Sigma \cup (\Omega \times \{0\})$.  For $X_0 \in \mathbb{R}^n$, $t_0 \in \mathbb{R}$ and $\tau >0$, we denote
\begin{eqnarray*}
 B_{\tau}(X_0) &:=& \left\{x \in \mathbb{R}^{n} : |X-X_0| < \tau \right\}, \\
 Q_{\tau}(X_0,t_0) &:=& B_{\tau}(X_0) \times (t_0-\tau^2, t_0 + \tau^2),\\
 Q^{-}_{\tau}(X_0,t_0) &:=& B_{\tau}(X_0) \times (t_0-\tau^2, t_0],
\end{eqnarray*}
and, for a set $K \subset \mathbb{R}^{n+1}$ and $\tau >0$,
$$\mathcal{N}_{\tau}(K) := \bigcup_{(X_0,t_0) \in K} Q_{\tau}(X_0,t_0) \quad \mathrm{and} \quad \mathcal{N}^{-}_{\tau}(K) := \bigcup_{(X_0,t_0) \in K} Q^{-}_{\tau}(X_0,t_0).$$

Let's start the section this section defining viscosity solutions of equation $\mathcal{L} v - v_t =f(X,t)$ in $\Omega_T$ and fix the notation.  For $p \in (1,+\infty)$, the operator 
\begin{equation} \label{OP1}
\mathcal{L} v \colon=\frac{1}{p}\Delta v + \frac{p-2}{p} \Delta^N_{\infty} v
\end{equation}
 is singular at $\{\nabla u =0\}$ and so the definition uses USC and LSC envelopes of the operator, see \cite{CIL92}.  Afterwards, we give some remarks to enlighten some basic features of solutions.   To this end,  given a function $h$ defined in a set $\mathcal{D}$, we need to introduce its \textit{upper semicontinuous envelope $h^{\star}$} and \textit{lower semicontinuous envelope $h_{\star}$} defined by
\begin{eqnarray*}
	h^{\star}(x,t) &\colon=& \lim_{r \downarrow 0} \sup \{h(y,t) : (y,t) \in \overline{Q_r(x,t)} \cap \mathcal{D}\}\\
	h_{\star}(x,t) &\colon=& \lim_{r \downarrow 0} \inf \{h(y,t) : (y,t) \in \overline{Q_r(x,t)} \cap \mathcal{D}\}\\
\end{eqnarray*}
as functions defined in $\overline{\mathcal{D}}$. The definition for viscosity solutions is the following.

\begin{definition}\label{Def1}
A continuous function $v$ is a viscosity subsolution (supersolution)  to the equation
\begin{equation} \label{OP2}
	\mathcal{L} v- v_t = g(X,t) \quad \textrm{in} \quad \Omega_T,
\end{equation}
if and only if  for all $(X_0,t_0) \in \Omega_T$ and  $\phi \in C^2(\Omega_T)$ such that $u-\phi$ attains a local maximum (minimum)  at $(X_0,t_0)$, then
\begin{enumerate}
\item $\mathcal{L} \phi(X_0,t_0) - \phi_t(X_0,t_0) \ge   g(X_0,t_0) \quad (\textrm{resp.} \,\, \le )$ if $\nabla \phi(X_0,t_0) \not= 0$
\\
\item$ \Delta \phi(X_0,t_0) + (p-2)\lambda_{\textrm{max}}(D^2 \phi(X_0,t_0))-\phi_{t}(X_0,t_0) \ge g(X_0,t_0)  \quad (\textrm{resp.} \,\, \le )$  if $\nabla \phi(X_0,t_0)=0$ and $p \ge 2$
\\
\item$ \Delta \phi(X_0,t_0) + (p-2)\lambda_{\textrm{min}}(D^2 \phi(X_0,t_0))-\phi_{t}(X_0,t_0) \ge g(X_0,t_0)  \quad (\textrm{resp.} \,\, \le ) $  if $\nabla \phi(X_0,t_0)=0$ and $1 < p <2$.
\end{enumerate}
We say that $u$ is a viscosity solution  of \eqref{OP2} in $\Omega_T$ if it is both a viscosity sub- and supersolution. 
\end{definition}

We will use the following two properties of the viscosity solutions of \eqref{OP2}.   
The first one is the comparison principle, which can be found in Theorem 3.2 in \cite{BGG}. 

\begin{proposition} \label{Prop1}
Let  $u$ lower-semicontinuous and $v$ upper-semicontinuous. Suppose that $v$ is a subsolution, and $u$ a supersolution to \eqref{OP2} with $g \in C(\overline{Q_1})$, $g >0$ and $1 \le p \le \infty$. Further, suppose that $v \le u$ on $\partial_p Q_1$. Then $v \le u$ in $\overline{Q_1}$.
\end{proposition}

The second one is the stability of viscosity solutions of \eqref{OP2}.

\begin{proposition}[Stability]\label{Prop2}
Let $\{u_k\}$ be a sequence of viscosity solution of \eqref{OP2} in $Q_{1}$, $u_{k}$ converge locally uniformly to $u$ and $g_k \to g$ locally uniformly in $Q_1$. Then $u$ is a viscosity solution of \eqref{OP2} in $Q_1$.
\end{proposition}

We need to clarify what is a Lipschitz function defined in a space-time domain.

\begin{definition}
Let $\mathcal{D}\subset \mathbb{R}^{n} \times  \mathbb{R}$. We say that $v\in\textrm{Lip}_{\loc}(1,1/2)(\mathcal{D})$ if, for every compact $K \Subset \mathcal{D}$, there exists a constant $C=C(K)$ such that
$$
|v(x,t) - v(y,s)| \le C \left(|x-y| + |t-s|^{\frac{1}{2}}\right),
$$
for every $(x,t), (y,s) \in K$.  If the constant $C$ does not depend on the set $K$ we say $v \in \textrm{Lip}(1,1/2)(\mathcal{D})$.
\end{definition}
We also define the $\textrm{Lip}(1,1/2)(\mathcal{D})$ seminorm in $\mathcal{D}$
$$
[v]_{\textrm{Lip}(1,1/2)(\mathcal{D})} := \sup_{(x,t),(y,s) \in \mathcal{D}} \frac{|v(x,t)-v(y,s)|}{|x-y| + |t-s|^{1/2}}
$$
and the $\textrm{Lip}(1,1/2)(\mathcal{D})$ norm in $\mathcal{D}$
$$
\|v\|_{\textrm{Lip}(1,1/2)(\mathcal{D})} := \|v\|_{L^{\infty}(\mathcal{D})} + [v]_{\textrm{Lip}(1,1/2)(\mathcal{D})}.
$$
\subsection{Hypothesis}

Throughout the paper, for the study of problem \eqref{Equation Pe}
\begin{itemize}
\item $f_{\varepsilon} (x,t) \in C^{1,\alpha}(\overline{{\Omega}_T})$, is non-increasing in $t$ and satisfies
\begin{equation} \label{E1}
0 < c_0 \le f_\varepsilon(x,t) \le c_1 < \infty  \quad \textrm{in} \quad \Omega_T
\end{equation}
and
\begin{equation} \label{E2}
\|\nabla f_{\varepsilon}\|_{\infty} \leq C.
\end{equation}
\item The Dirichlet data $0\leq \varphi (x,t) \in C^{1,\alpha}(\partial_p\Omega_T)$, is non-decreasing in $t$ and satisfies $\varphi(x,0)=0$.
\end{itemize}

These are the same hypothesis assumed in \cite{RMR}.

\subsection{Existence and optimal regularity}

Our existence theorem relies on a singularly perturbed analysis. To this end, we shall define the perturbed term $\zeta_{\varepsilon} : \mathbb{R}_{+} \rightarrow \mathbb{R}_{+}$ satisfying
$$
	0 \le \zeta_{\varepsilon}(s) \le \frac{1}{\varepsilon} \chi_{(0,\varepsilon)}(s), \quad \forall \, s \in \mathbb{R}_{+}
$$
and the corresponding problem \eqref{Equation Pe}.  For example, it can be built as an approximation of unity
$$
	\zeta_{\varepsilon}( s) := \frac{1}{\varepsilon} \zeta \left(\frac{s}{\varepsilon}\right),
$$
where  $\zeta$ is a nonnegative smooth real function with $\supp ~\zeta = [0,1]$, such that
$$
\|\zeta\|_\infty\le 1 \quad \textrm{and} \quad \int_{\mathbb{R}} \zeta (s) \, ds < \infty.
$$

This existence result is based on Perron's method. We state before following theorem independently
of the \eqref{Equation Pe} context, since it may be of independent interest.

\begin{proposition}[Perron's Method] \label{Perron} 
Let $g: [0, \infty) \to \mathbb{R} $ be a bounded and Lipschitz function and $\mathcal{F}: \Omega_T \times  \mathbb{R}^n \times Sym(n) \to \mathbb{R}$ a degenerate elliptic operator satisfying the following monotonicity condition 
\begin{equation}\label{monot}
\mathcal{F}(x,t,\vec{p},N) \leq \mathcal{F}(x,t,\vec{p},M) \quad \mbox{whenever} \quad N \leq M,
\end{equation}
for any $\vec{p} \in \mathbb{R}^n$ and $N,M \in Sym(n)$. If the equation 
\begin{equation}\label{perroneq}
	u_t- \mathcal{F}(x,t, \nabla u,D^2 u) = g(u) + f(X,t)
\end{equation}
with $f \in C(\Omega_T)$, admits subsolution and supersolution  $\underline{u}, \overline{u} \in C^{0}(\overline{\Omega_T})$ respectively, and $\underline{u}=\overline{u} = \varphi \in W^{2,\infty}(\partial_p\Omega_T)$, then given the set of functions
$$
	\mathscr{S} := \left \{ w \in C(\overline{\Omega_T}) \suchthat w \text{ is a supersolution to } \eqref{perroneq}, \text{ and } \underline{u} \le w \le \overline{u} \right \}, 
$$ 
the function
\begin{equation}\label{inf}
	v(x,t) := \inf_{w \in \mathscr{S}} w(x,t) 
\end{equation}
is a continuous viscosity solution to \eqref{perroneq}, safisfying $u=\varphi$ in $\partial_p \Omega_T$.  
\end{proposition}

\begin{proof}
Using Proposition \ref{Prop1}, the proof follows exactly as the one of \cite[Theorem 3.1]{RMR}.
\end{proof}

In order to prove the existence result of the problem \eqref{Equation Pe}, we
choose to approximate the equation \eqref{Equation Pe} with a regularized problem. For $\delta >0$, let $u_{\varepsilon,\delta}$ be smooth  and satisfying that
 \begin{equation}\label{sol3}
     \left \{
        \begin{array}{rclcl}
            a^{\delta,p}_{ij}(\nabla u)u_{ij} - u_t&=& \zeta_{\varepsilon}(u) + f_{\varepsilon} & \textrm{ in }& \Omega_T \\
            u(x,0)&=& \varphi\ & \textrm{ on } & \partial_p \Omega_T,
        \end{array}
    \right.
\end{equation}
 where  
 \begin{equation} \label{coef}
 	a^{\delta,p}_{ij}(\vec \eta) \colon= \frac{1}{p} \delta_{ij} + \frac{p-2}{p} \frac{\eta_i \eta_j}{|\vec \eta|^2 + \delta}.
 \end{equation}
   Note that, $(a^{\delta,p}_{ij})_{i,j}$ is uniformly parabolic with
$$
	\lambda |\xi|^2 \le a^{\delta,p}_{ij}(\vec \eta) \xi_i \xi_j \le \Lambda |\xi|^2, 
$$
where $\Lambda = \max \left( \frac{p-1}{p}, \frac{1}{p}\right)$ and $\lambda = \min \left(\frac{p-1}{p}, \frac{1}{p}\right)$. 
 Since $p >1$, it follows from classical quasilinear equation theory (see e.g. \cite[Theorem 4.4]{O}) and the Schauder estimates that, for each $\delta>0$, there exists a unique solution $u_{\varepsilon,\delta} \in C^{3,\beta}(\Omega_T) \cap C(\overline{\Omega_T})$ of \eqref{sol3} with $p>1$. 
\begin{theorem}
Let $u_{\varepsilon, \delta} : \overline{\Omega} \times (0,T) \rightarrow \mathbb{R}$ be a viscosity solution of the regularized problem \eqref{sol3}. Then there exists a viscosity subsolution $u^{\star} : \overline{\Omega} \times (0,T) \rightarrow \mathbb{R}$ and a viscosity supersolution $u_{\star} : \overline{\Omega} \times (0,T) \rightarrow \mathbb{R}$ of the original problem \eqref{Equation Pe}, which have the form
\begin{eqnarray*}
	u^{\star}_{\varepsilon}(x,t) &\colon=& \lim_{\delta \to 0} \inf \left\{u_{\varepsilon, \mu}(\xi) : \, \xi \in B_{\delta}(x,t), \, 0 < \mu< \delta\right\}\\
	u_{\star,\varepsilon}(x,t) &\colon=& \lim_{\delta \to 0} \sup \left\{u_{\varepsilon, \mu}(\xi) : \, \xi \in B_{\delta}(x,t), \, 0 < \mu< \delta\right\}
\end{eqnarray*} 
In particular, the problem \eqref{Equation Pe} has a viscosity solution $\{u_{\varepsilon}\}$. Moreover,  there exists a constant $\Upsilon= \Upsilon(n,\lambda,\Lambda,c_0)$, such that
\begin{equation} \label{R1}
 0 \le u_{\varepsilon} \le \Upsilon.
\end{equation}
\end{theorem}
\begin{proof}
For a proof of first part of the Theorem, we refer the reader to \cite{Does}. The existence of viscosity solution to \eqref{Equation Pe} follows from  Proposition \ref{Perron}.  To prove \eqref{R1}, let $v^{\varepsilon}:= u_{\varepsilon} -\|\varphi\|_\infty$. Note that $v^{\varepsilon} \le 0$ on $\partial_p\Omega_T$ and 
\begin{eqnarray*}
    a^{\delta,p}_{ij}(\nabla v^{\varepsilon})v^{\varepsilon}_{ij} -\partial_tv^{\varepsilon}=   a^{\delta,p}_{ij}(\nabla u_{\varepsilon})D_{ij} u_{\varepsilon} -\partial_tu_{\varepsilon}\ge c_0.
\end{eqnarray*}
This means that $v_{\varepsilon} \in\underline{\mathscr{S}}(\frac{\lambda}{n}, \Lambda,c_0)$. The ABP estimate (\cite[Theorem 3.14]{Wang}) then implies 
$$\displaystyle\sup_{\Omega_T} (v_{\varepsilon})^{+} \le C(\lambda, \Lambda, n, c_0).$$
Thus, $u_{\varepsilon} \le \|\varphi\|_{\infty} + C(\lambda, \Lambda, n, c_0) =: \Upsilon$.

In order to prove the nonnegativity of $u_{\varepsilon}$ we assume the contrary, i.e. that $A_\varepsilon := \{(x,t) \in\Omega_T  : u_{\varepsilon}(x,t) < 0\}\neq\emptyset$. Since $\zeta_\varepsilon$ is supported in
$[0,\varepsilon]$, then
$$
  a^{\delta,p}_{ij}(\nabla u_{\varepsilon})D_{ij} u_{\varepsilon}   -\partial_tu_{\varepsilon} = f_{\varepsilon}\le c_1\,  \textrm{ in }\, A_\varepsilon,
$$
which means that $u_\varepsilon\in\mathscr{\overline{S}}(\frac{\lambda}{n},\Lambda,c_1)$. Another application of the ABP estimate provides that $u_\varepsilon \ge 0$ in $A_\varepsilon$, which is a contradiction.

\end{proof}

In order to pass to the limit as $\varepsilon \to 0^{+}$ to obtain a solution of \eqref{P1} we need sharp estimates uniform in $\varepsilon \in (0,1]$. Our main contribution is the following.

\begin{theorem} \label{tel4.1}
Let $u_{\varepsilon}$ be a viscosity solution of $\eqref{Equation Pe}$. Let $K \subset\Omega_T$ be compact and $\tau >0$ be such that $\mathcal{N}_{2 \tau}(K) \subset\Omega_T$. Then there exists a constant $L=L(\tau,\lambda,\Lambda,c_0,c_1, \|\nabla f\|_{\infty}, \|\varphi\|_{\infty}, K)$ (in particular, not depending on $\varepsilon$) such that
$$
	\| u_{\varepsilon} \|_{\textrm{Lip}(1,1/2)} \le L.
$$
\end{theorem}

Theorem \ref{tel4.1} gives us the necessary compactness to pass to the limit as $\varepsilon \to 0^+$ obtaining a viscosity solution of \eqref{P1}. As a consequence, we also have sharp regularity.

\begin{theorem} \label{te4.1}
Let $u$ be a viscosity solution of $\eqref{P1}$. Let $K \subset\Omega_T$ be compact and $\tau >0$ be such that $\mathcal{N}_{2 \tau}(K) \subset\Omega_T$. Then there exists a constant $L=L(\tau,\lambda,\Lambda,c_0,c_1, \|\nabla f\|_{\infty}, \|\varphi\|_{\infty}, K)$ such that
$$
	\|u\|_{\textrm{Lip}(1,1/2)(K)}\le L.
$$
\end{theorem}

This regularity result and some other properties of the free boundary problem \eqref{P1} are proved in Section \ref{s4}. 


\subsection{Some properties of the free boundary}

In this section we establish the exact growth of the solution near the free boundary, from which we deduce the porosity of its time level sets.
\begin{definition}\label{d5.1}
A set $E\subset\mathbb{R}^n$ is called porous with porosity $\delta>0$, if there exists $R>0$ such that
$$
   \forall x\in E, \,\,\,\forall r\in(0,R),\,\,\,\exists y\in\mathbb{R}^n\,\textrm{ such that }\,B_{\delta r}(y)\subset B_r(x)\setminus E.
$$
\end{definition}
A porous set of porosity $\delta$ has Hausdorff dimension not exceeding $n-c\delta^n$, where $c=c(n)>0$ is a constant depending only on $n$. In particular, a porous set has Lebesgue measure zero.

The main result of this section is as follows.
\begin{theorem}\label{t6.1}
Let $u$ be a solution of \eqref{P1}. Then, for every compact set $K \subset \Omega_T$ and every $t_0 \in (0,T)$, the set 
$$\partial\{u > 0\}\cap K \cap\{t=t_0\}$$ 
is porous in $\mathbb{R}^n$, with porosity depending only on $\Upsilon$ and $\textrm{dist}(K,\partial_p \Omega_T)$. In other words,
$$
	\mathcal{L}^{n}(\partial\{u > 0\}\cap K \cap\{t=t_0\})=0.
$$
\end{theorem}

The proof is based on the exact growth of the solution of the problem \eqref{P1} near the free boundary. This result is proved in Section \ref{s6}.
\begin{theorem}
Assume $p \in (1,+\infty)$ and let $u$ be a viscosity solution of the problem \eqref{P1}. Then there exist two positive constants $d_0=d_0(n,p,c_0)$ and $D_0 = D_0(n,p,c_0,c_1)$ such that for every compact set $K \Subset \Omega_T$, $(X_0,t_0) \in (\partial \{u>0\}) \cap K$, the following estimates hold
$$
	d_0 r^2 \le \sup_{B_r(X_0)} u(\cdot, t_0) \le D_0 r^2.
$$
\end{theorem}

\section{Uniform Estimates in time and space for the perturbed problem}\label{s4}

\hspace{0.3cm} This section discusses regularity for the viscosity solution $u_{\varepsilon}$ of the singular perturbation problem \eqref{Equation Pe} for $p \in (1,+\infty)$.  First of all we show Lipschitz continuity of the viscosity solution $u_{\varepsilon}$ with respect to $x$ using a Bernstein type argument.  The strategy to show Lipschitz regularity  is based on the works \cite{Does,RMR} but it turns out that the result is not true for $p=1$ since the constant $\bar L$ (see Proposition \ref{p4.22}) blows up for $p \to 1$. Finally, we will show that bound on the gradients implies limitation in the 
seminorm $\textrm{Lip} \left(1,1/2\right)$. 



\subsection{Uniform spatial regularity.}

In order to prove Theorem \ref{tel4.1}, we choose a regularized problem given by \eqref{sol3}.   We start with the uniform Lipschitz regularity in the spatial variables.

\begin{proposition}\label{p4.22}
If  $\{u_{\varepsilon, \delta}\}$ is a family of solutions of \eqref{sol3}, and $\eqref{E1}$-$\eqref{E2}$ hold, then there exists a constant $\bar L=\bar L(\tau,\lambda,\Lambda,c_0,c_1, \|\nabla f\|_{\infty}, \|\varphi\|_{\infty}, K)>0$, independent of $\delta>0$ and $\varepsilon\in(0,1)$, such that
\begin{equation} \label{M1}
| \nabla u_{\varepsilon,\delta}(x,t)| \le \bar L  \left(1+\frac{1}{\textrm{dist}((x,t), \partial_p \Omega_T)^2}\right),\,\, \forall(x,t) \in \Omega_T.
\end{equation}
\end{proposition}
\begin{proof}
Fix $\varepsilon >0$ and for each  $\delta>0$, let $u_{\varepsilon, \delta}=u^{\delta}$ (for simplicity) be a smooth viscosity solution to \eqref{sol3}.To show \eqref{M1}, initially, since $\zeta_\varepsilon=0$ in $\{u^{\delta} >\varepsilon\}$, we conclude from up to the boundary parabolic regularity theory (see \cite[Theorem 4.19]{Wang} and  \cite[Theorem 2.5]{Wang2} )  that
$$
| \nabla u^{\delta}|\leq C(\|u^{\delta}\|_\infty+\|f_\varepsilon\|_{n+1}+\|\varphi\|_\infty),
$$
in this region, where $C$ does not depend on $\varepsilon$ and $\delta>0$. The result then follows from $\eqref{E2}$ and \eqref{R1}, passing to the limit $\delta \to 0$. To show a Lipschitz estimate with respect to $x$ in $\{u^{\delta} \le \epsilon\}$, we define $v(x,t) = \left(|Du^{\delta}|^2 + \delta\right)^{1/2}$ and consider the function
$$
	w(x,t) = \xi(x,t) \cdot v(x,t) +\mu u^{\delta}(x,t)^2,
$$
where $\mu = \frac{\Gamma}{2 \varepsilon^2}$, for some $\Gamma \ge 0$, $0 \le \xi \le 1$ is a positive smooth function that vanishes on the parabolic boundary $\partial_p \Omega_{T}$. Let $(x_0,t_0) \in \{u^{\delta} \le \varepsilon\}$ be a point where $w$ takes its maximum in $\Omega_T \cup \partial_p \Omega_T$.  We can assume without loss of generality that
\begin{equation}\label{cond}
	|\nabla u^{\delta}(x_0,t_0)| \ge 1.
\end{equation}
In fact, if $|\nabla u^{\delta}| <1$, then if $0 < \delta \le 1/2$,
\begin{eqnarray*}
|\nabla u^{\delta}(x,t)| &\le& v(x,t) \le w(x,t) \le w(x_0,t_0) \\
&=& \xi(x_0,t_0) v(x_0,t_0) + \mu u^{\delta}(x_0,t_0)^2 \\
&\le& \|\xi\|_{\infty} \sqrt{1+\delta} + \varepsilon^2 \mu \le \sqrt{1+\delta} + \Gamma  \le 2+\Gamma \colon= \bar L.
\end{eqnarray*}

First we suppose that $(x_0,t_0) \not\in \partial_p \Omega_T$. At that point $D^2 w(x_0,t_0)$ is negative definite. We define the coeficient matrix $a^{p,\delta}_{ij}$ as in \eqref{coef}.
Since the matrix $(a^{p,\delta}_{ij}(Du^{\delta}(x,t)))_{ij}$ is negative semi-definite  for all points $(x,t)$ we have
$$
	\sum_{i,j}a^{\delta,p}_{ij}(\nabla u^{\delta}(x_0,t_0)) w_{ij}(x_0,t_0) \le 0
$$
Then,
\begin{eqnarray}
0 &\le& w_t(x_0,t_0) - \sum_{i,j}a^{\delta,p}_{ij}(\nabla u^{\delta}(x_0,t_0))w_{ij}(x_0,t_0)\nonumber\\
&=& \xi \left(v_t - \sum_{i,j} a^{p,\delta}_{ij}(Du^{\delta})v_{ij}\right) + v \left(\xi_t -\sum_{i,j} a^{p,\delta}_{ij}(Du^{\delta}) \xi_{ij}\right) \nonumber\\
&+& 2 \lambda u^{\delta} \left(u^{\delta}_{t} -\sum_{i,j} a^{p,\delta}_{ij}(Du^{\delta}) u^{\delta}_{ij}\right)-2 \sum_{i,j} a^{p,\delta}_{ij}(Du^{\delta}) \xi_j v_{i}\nonumber \\
&-& 2 \lambda \sum_{i,j} a^{p,\delta}_{ij}(D u^{\delta}) u^{\delta}_{i} u^{\delta}_{j}  \label{delta1}
\end{eqnarray}
holds at $(x_0,t_0)$. We estimate each term separately to obtain the desired inequality. The third term on the right hand side is 
\begin{equation} \label{delta4}
	2 \mu u^{\delta} \left( u^{\delta}_{t} -\sum_{i,j} a^{p,\delta}_{ij}(Du^{\delta}) u^{\delta}_{ij} \right)=	-2\mu u^{\delta} \left ( f_{\epsilon} + \zeta_{\epsilon}(u^{\delta}) \right)
\end{equation}
 because $u^{\delta}$ is a classical solution of the approximating problem. Now, we consider
$$
	 \xi \left(v_t - \sum_{i,j} a^{p,\delta}_{ij}(Du^{\delta})v_{ij}\right) .
$$
We differentiate the equation \eqref{delta4} with respect to $x_k$,  multiply $u^{\delta}_{tk}$ with $\frac{u^{\delta}_{k}}{v}$, and sum from $1,\ldots, n$ to obtain
\begin{eqnarray*}
	v_t &=& \frac{1}{p} \frac{1}{v} \sum_{i,k} u^{\delta}_k u^{\delta}_{iik} + \frac{p-2}{p} \frac{1}{v^3} \sum_{i,j,k} u^{\delta}_{i} u^{\delta}_{j} u^{\delta}_{k} u^{\delta}_{ijk}\\
	&+& \frac{p-2}{p} \frac{2}{v^3} \sum_{i,j,k} u^{\delta}_{i} u^{\delta}_{ij} u^{\delta}_{k} u^{\delta}_{jk} - \frac{p-2}{p} \frac{2}{v^5} \left(\sum_{i,j} u^{\delta}_i u^{\delta}_{j} u^{\delta}_{ij}\right)^2 \\
	&-& \frac{1}{v} \sum_{k} \left\{D_k f_{\epsilon}u^{\delta}_k + \epsilon^{-2} \zeta' (u^{\delta}_{k})^2\right\}
\end{eqnarray*}
The second derivatives of $v$ are
$$
v_{ij} =\frac{1}{v} \sum_{k} u^{\delta}_{ik} u^{\delta}_{jk} + \frac{1}{v} \sum_{k} u^{\delta}_{k} u^{\delta}_{ijk} - \frac{1}{v^3} \sum_{k} \left( u^{\delta}_k u^{\delta}_{ik}\right) \sum_{\ell} \left(u^{\delta}_{\ell}u^{\delta}_{j \ell}\right),
$$
and thus we have by a straightforward calculation
\begin{eqnarray*}
v_t - \sum_{i,j} a^{p,\delta}_{ij}(Du^{\delta}) v_{ij} &\le& -\frac{1}{p} \frac{1}{v} \sum_{i,k} (u^{\delta}_{ik})^2 - \frac{p-2}{p} \frac{1}{v^5} \left(\sum_{i,j} u^{\delta}_i u^{\delta}_j u^{\delta}_{ij}\right)^2\\
&+& \frac{p-1}{p}  \Lambda \frac{1}{v^3} |Dv|^2 v^2 - \frac{1}{v} \sum_{k} D_k f_{\epsilon}u^{\delta}_k - \frac{1}{v}\epsilon^{-2} \zeta' |Du^{\delta}|^2
\end{eqnarray*}
Note that
$$
	-\frac{1}{p} \frac{1}{v} \sum_{i,k} (u^{\delta}_{ik})^2 - \frac{p-2}{p} \frac{1}{v^5} \left(\sum_{i,j} u^{\delta}_i u^{\delta}_j u^{\delta}_{ij}\right)^2 \le 0
$$
since
$$
	\frac{1}{p} \sum_{i,k} (u^{\delta}_{ik})^2 + \frac{p-2}{p} \left(\sum_{i,j}u^{\delta}_{ij} \frac{u^{\delta}_{i}}{v}\frac{u^{\delta}_{j}}{v}\right)^2 \ge \lambda \sum_{i,k} (u^{\delta}_{ik})^2 \ge 0
$$
and we have

$$
	v_t - \sum_{i,j} a^{p,\delta}_{ij}(Du^{\delta}) v_{ij} \le \frac{1}{v} \left( \Lambda |Dv|^2 + v \cdot \|D f_{\epsilon}\|_{\infty}  + \epsilon^{-2} \zeta' |Du^{\delta}|^2 \right).
$$

In order to estimate the fourth term in \eqref{delta1} we use the fact that $w_i = \xi_i v + \xi v_i + 2\lambda u^{\delta}u^{\delta}_i=0$ in $(x_0,t_0)$ and get 
\begin{eqnarray}
\xi \frac{|Dv|^2}{v} &\le& \frac{v}{\xi} \left(|D \xi|^2 + 4 \mu |u^{\delta}| |D \xi| + 4 (\mu u^{\delta})^2\right)\nonumber\\
&\le& \frac{5v}{\xi} \left( |D \xi|^2 + (\mu u^{\delta})^2\right). \label{delta7}
\end{eqnarray}
Hence
\begin{eqnarray*}
	-2 \sum_{i,j} a^{p,\delta}_{ij}(Du^{\delta}) \xi_j v_i &\le& \frac{2v}{\xi} \frac{1+|p-2|}{p} |D \xi|^2 + \frac{4v}{\xi} \frac{1+|p-2|}{p} \mu u^{\delta} |D \xi|\\
	&\le& 4\frac{1+ |p-2|}{p} \frac{v}{\xi} \left(|D \xi|^2 + (\mu u^{\delta})^2\right)
\end{eqnarray*}
Moreover, using Young's inequality $a \cdot b \le \frac{1}{2} (a^2 +b^2)$ we obtain an estimate for the second term in \eqref{delta1}
\begin{eqnarray*}
	v \left(\xi_t - \sum_{i,j} a^{p,\delta}_{ij}(Du^{\delta}) \xi_{ij}\right) &\le& v |\xi_t| + \frac{n}{p} v |D^2 \xi| + \frac{|p-2|}{p} v |D^2 \xi| \\
	&\le& \eta \lambda v^2 + \frac{1}{4 \lambda \eta} \left[|\xi_t| + \frac{n+|p-2|}{p} |D^2 \xi| \right]^2
\end{eqnarray*}
with $a \colon= \sqrt{2 \eta \mu} \cdot v$ and $b \colon= \sqrt{\frac{1}{2 \eta \mu}} \left(|\xi_t|+ \frac{n+|p-2|}{p} |D^2 \xi|\right)$ for some $\eta >0$.  With all these inequalities and using \eqref{delta7} we obtain for the fifth term in \eqref{delta1}
\begin{eqnarray*}
	2 \lambda \sum_{i,j} a^{p,\delta}_{ij}(Du^{\delta}) u^{\delta}_i u^{\delta}_j &=& 2 \lambda |Du^{\delta}|^2 \left(\frac{1}{p} + \frac{p-2}{p} \frac{|Du^{\delta}|^2}{|Du^{\delta}|^2 + \delta}\right)\\
	&\le& \frac{1+|p-2|}{p}\frac{5v}{\xi} \left(|\nabla \xi|^2 + (\lambda u^{\delta})^2\right) + \eta \mu v^2 \\
	&+& \frac{1}{4 \eta \mu} \left(|\xi_t| + \frac{n+p-2}{p} |D^2 \xi|\right)^2 \\
	&+& 4 \frac{1+|p-2|}{p} \frac{v}{\xi} \left(|\nabla \xi|^2 + (\mu u^{\delta})^2\right) - 2 \mu u^{\delta} \left(\zeta_{\varepsilon}(u^{\delta}) + f_{\varepsilon}\right)\\
	&-& \frac{\xi}{v} \sum_{k} \left\{D_k f_{\varepsilon} \cdot u^{\delta}_k + \varepsilon^{-2} \zeta' (u^{\delta}_k)^2\right\} \\
	&\le& 10 \frac{v}{\xi} \left(|\nabla \xi|^2 + (\mu u^{\delta})^2\right) + \eta \mu v^2 \\
	&+& \frac{1}{4\eta \mu} \left(|\xi_t| + \frac{n+|p-2|}{p} |D^2 \xi|\right) -2 \mu u \left(\zeta_{\varepsilon}(u^{\delta}) + f_{\varepsilon}\right)\\
	&-& \frac{\xi}{v} \sum_{k} \left\{D_k f_{\varepsilon} \cdot u_k + \varepsilon^{-2} \zeta' u^2_k\right\} .
\end{eqnarray*}
Now, using the Young's inequality with $a=10 \frac{1}{\sqrt{2 \eta \mu}} \frac{(|\nabla \xi|^2 + (\lambda u^{\delta})^2)}{\xi}$ and $b = \sqrt{2 \eta \mu} \cdot v$ and noting that, if $|\nabla u^{\delta}(x_0,t_0)| \ge 1$ and $0 < \delta \le 1/2$ , we can bound from below the expression
\begin{eqnarray*}
	\left(\frac{8}{9} \frac{p-1}{p}\right) \mu v^2 \le 2 \mu |\nabla u^{\delta}|^2 \left(\frac{1}{p} + \frac{p-2}{p} \frac{|\nabla u^{\delta}|^2}{|\nabla u^{\delta}|^2 + \delta}\right).
\end{eqnarray*}
Thus, we get that
\begin{eqnarray}
\left(\frac{8}{9} \frac{p-1}{p} - 2 \eta \right) \xi^2  v^2 &\le& \frac{25}{\eta \mu^2} \left(\left(|\nabla \xi|^2 + (\mu u^{\delta})^2\right)^2 \right)\nonumber\\
&+& \frac{1}{4 \eta \mu^2} \xi^2 \left(|\xi_t| + \frac{n+|p-2|}{p} |D^2 \xi |\right)^2+ 2 \xi^2 u^{\delta} \left(\zeta_{\varepsilon}(u^{\delta}) + f_{\varepsilon}\right)\nonumber\\
	&+& \frac{\xi^3}{v}\mu^{-1} \sum_{k} \left\{D_k f_{\varepsilon} \cdot u^{\delta}_k + \varepsilon^{-2} \zeta' (u^{\delta}_k)^2\right\}. \label{Equ1}
\end{eqnarray}
The coefficient $\left(\frac{8}{9} \frac{p-1}{p} - 2 \eta\right)$ is supposed to be positive because we divide by it and want to preserve the same direction of the inequality. Note that, if $p \to 1$ the coefficient becomes negative. But for fixed $p>1$, we can choose $\eta=\eta(p)>0$ such that 
$$
	\left(\frac{8}{9} \frac{p-1}{p} - 2 \eta\right) \ge \frac{5}{9} \frac{p-1}{p} \colon= \frac{1}{c(p)}.
$$ 
Now, note that,
\begin{eqnarray}
2 \xi^2 u^{\delta} \left(\zeta_{\varepsilon}(u^{\delta}) + f_{\varepsilon}\right) &\ge& 0. \label{Eq2}
\end{eqnarray}
and
\begin{eqnarray}
	 \frac{\xi^3}{v}\mu^{-1} \sum_{k} \left\{D_k f_{\varepsilon} \cdot u^{\delta}_k + \varepsilon^{-2} \zeta' (u^{\delta}_k)^2\right\} &=&  \frac{\xi^2}{v}\mu^{-1}\left( \sum_{k} D_k f_{\varepsilon} \cdot u^{\delta}_k + \varepsilon^{-2} \zeta' |\nabla u^{\delta}|^2 \right) \nonumber\\
	 &\le& \frac{\xi^2}{v} \mu^{-1} \left(\|\nabla f_{\varepsilon}\| \cdot |\nabla u^{\delta}| + \varepsilon^{-2} |\zeta'||\nabla u^{\delta}|^2\right) \nonumber\\
	 &\le& \frac{2}{\Gamma} \xi^2 \left(\|\nabla f_{\varepsilon}\| \cdot \frac{|\nabla u^{\delta}|}{v} + 2 \sup|\zeta'| \frac{|\nabla u^{\delta}|^2}{v}\right) \nonumber\\
	 &\le& \frac{2}{\Gamma} \xi^2 \left(\|\nabla f_{\varepsilon}\| + 2 \sup |\zeta'| |\nabla u^{\delta}|^2 \right). \label{Eq3}
\end{eqnarray}
Substituting \eqref{Eq2} and \eqref{Eq3} into \eqref{Equ1}, we obtain
\begin{eqnarray*}
\xi^2 v^2 &\le& \frac{c(p)}{\mu^2} \left[\left(|\nabla \xi|^2 + (\mu u^{\delta})^2\right) + \xi \left(|\xi_t| + \frac{n+|p-2|}{p} |D^2 \xi|^2 \right)\right]^2 \\
&+& \frac{2 \cdot c(p)}{\Gamma} \xi^2 \|\nabla f_{\varepsilon}\| + \frac{2 \cdot c(p)}{\Gamma} \sup |\zeta'| \xi^2  v^2.
\end{eqnarray*}
Thus,
\begin{eqnarray*}
\xi^2 \left(1- \frac{2 \cdot c(p)}{\Gamma} \sup |\zeta'|\right) v^2 &\le& \frac{c(p)}{\mu^2} \left[\left(|\nabla \xi|^2 + (\mu u^{\delta})^2\right) + \xi \left(|\xi_t| + \frac{n+|p-2|}{p} |D^2 \xi|^2 \right)\right]^2  \\&+& 
 + \frac{2 \cdot c(p)}{\Gamma}  \|\nabla f_{\varepsilon}\| .
\end{eqnarray*}
Therefore, we can choose $\Gamma =\Gamma(p)>0$  such that 
$$
	1 - \frac{2 \cdot c(p)}{\Gamma} \|\zeta'\|_{\infty} \ge \frac{1}{2}.
$$
This leads to the following inequality, at $(x_0,t_0)$
\begin{eqnarray*}
\xi^2 v^2 &\le& \frac{c(p)}{\mu^2} \left[\left(|\nabla \xi|^2 + (\mu u^{\delta})^2\right) + \xi \left(|\xi_t| + \frac{n+|p-2|}{p} |D^2 \xi|^2 \right)\right]^2+ \kappa(p) \\
\end{eqnarray*}
 where 
$
	\kappa(p) \colon= \frac{2 \cdot c(p)}{\Gamma}  \|\nabla f_{\varepsilon}\|_{\infty} .
$
Note that the constant $c(p)$  blows up for $p \to 1$. 
 Now, fix $(x,t) \in \Omega_T$ and choose $\xi \in (0,1)$ such that $\xi(x_0,t_0)=1$ and
$$
	\max \left\{\|D^2 \xi\|_{\infty},\|\nabla \xi\|_{\infty}, \|\xi_t\|_{\infty}\right\} \le \frac{1}{\textrm{dist}((x,t), \partial_p \Omega_T)}.
$$
Then
\begin{eqnarray*}
	|\nabla u^{\delta}(x,t)| &\le& w(x,t) \le w(x_0,t_0) = \xi(x_0,t_0) v(x_0,t_0) + \frac{\Gamma}{2 \varepsilon^2} (u^{\delta}(x_0,t_0))^2 \\
	&\le& \frac{C(p,n)}{\mu} \left(\|D^2 \xi\|_{\infty} + \Gamma + \|\nabla \xi\|^2_{\infty} + \|\xi_t\|_{\infty}\right)+\kappa(p)\\
	&\le& C(n,p) \left(1 + \frac{1}{\textrm{dist}((x,t), \partial_p \Omega_T)} +\frac{1}{\textrm{dist}((x,t), \partial_p \Omega_T)^2}\right).
\end{eqnarray*}
We consider two cases. First we suppose that $\textrm{dist}((x,t), \partial_p \Omega_T) \le 1$. This implies
\begin{eqnarray*}
	|\nabla u^{\delta}(x,t)| &\le& C(n,p) \left(1 + \frac{2}{\textrm{dist}((x,t), \partial_p \Omega_T)^2}\right)\\
	&\le& C(n,p) \left(1 + \frac{1}{\textrm{dist}((x,t), \partial_p \Omega_T)^2}\right)
\end{eqnarray*}
For the sake of simplicity let $C(n,p)$ be a generic constant. We obtain, for $\textrm{dist}((x,t), \partial_p \Omega_T) \ge 1$,
\begin{eqnarray*}
	|\nabla u^{\delta}(x,t)| &\le& C(n,p) \left(2 + \frac{1}{\textrm{dist}((x,t), \partial_p \Omega_T)^2}\right)\\
	&\le& C(n,p) \left(1 + \frac{1}{\textrm{dist}((x,t), \partial_p \Omega_T)^2}\right)
\end{eqnarray*}
Finally we treat the case when the maximum point $(x_0,t_0)$ of $w$ is attained on the
parabolic boundary $\partial_p \Omega_T$ , then
\begin{eqnarray*}
	|\nabla u^{\delta}(x,t)| &\le& v(x,t) \le w(x,t) \le w(x_0,t_0) = \mu (u^{\delta}(x_0,t_0))^2\\
	&\le& \|\varphi\|_{\infty} \le \bar C(p,n,\|\varphi\|_{\infty}) \left(1 + \frac{1}{\textrm{dist}((x,t), \partial_p \Omega_T)^2}\right)
\end{eqnarray*}
because $\xi \equiv 0$ on $\partial_p \Omega_T$.
\end{proof}

Using the Proposition \ref{p4.22}, we obtain the following result 

\begin{corollary}\label{p4.1}
If  $\{u_{\varepsilon}\}_{\varepsilon >0}$ is a family of solutions of \eqref{Equation Pe}, and $\eqref{E1}$-$\eqref{E2}$ hold, then there exists a constant $\bar L=L(\tau,\lambda,\Lambda,c_0,c_1, \|\nabla f\|_{\infty}, \|\varphi\|_{\infty}, K)>0$, independent of  $\varepsilon\in(0,1)$, such that
$$
| \nabla u_{\varepsilon}(x,t)| \le \bar L  \left(1+\frac{1}{\textrm{dist}((x,t), \partial_p \Omega_T)^2}\right),\,\, \forall(x,t) \in \Omega_T.
$$
\end{corollary}
\begin{proof}
Fix $\epsilon >0$. Initially, note that  the approximating functions $u_{\varepsilon, \delta}$ converge locally uniformly to the viscosity solution $u_{\varepsilon}$ of the singular perturbation problem (see Section 3.2 in \cite{Does})
$$
	\mathcal{L}u_{\varepsilon} - \partial_t u_{\varepsilon}=f_{\epsilon} + \zeta_{\epsilon}(u).
$$
as $\delta \to 0$. Since $u_{\varepsilon, \delta}$ is uniformly locally Lipschitz continuous, there exists for every $x \in \Omega$ a neighborhood $U_x \subset \Omega$ of $x$ and a constant $L>0$ such that
\begin{eqnarray*}
	|u_{\varepsilon}(y,t) - u_{\varepsilon}(z,t)| &\le& |u_{\varepsilon}(y,t)-u_{\varepsilon, \delta}(y,t)| + |u_{\varepsilon,\delta}(y,t) - u_{\varepsilon,\delta}(z,t)|\\
	& +& |u_{\varepsilon, \delta}(z,t)- u_{\varepsilon}(z,t)| 
	\le 2 \tilde{\epsilon}(\delta) + L\cdot |y-z|
\end{eqnarray*}
for all $y,z \in U_x$ and fixed $t$, where $\tilde{\epsilon} \to 0$ for $\delta \to 0$ and $L$ independent of $\varepsilon$ and $\delta$. Then Rademacher's Theorem, e.g., stated in \cite{16}, provides that the gradient $\nabla u_{\varepsilon}(x,t)$ exists almost everywhere in $\Omega_T$. Let $(x,t) \in \Omega_T$ be a point where the gradient of $u$ exists and $B_{r}(x,t) \Subset \Omega_T$ be a ball of radius $r>0$ around $(x,t)$, then by Proposition \ref{p4.22}
$$
	\left| \nabla u_{\varepsilon, \delta}(y,s)\Big|_{B_r(x,t)}^{} \right| \le \bar C \left( 1 + \frac{1}{\min_{(y,s) \in B_r(x,t)} (\textrm{dist}((y,s), \partial_p \Omega_T))^2}\right)
$$
for all $(y,s) \in B_r(x,t)$. Passing to the limit $\delta \to 0$;
$$
	\left| \nabla u_{\varepsilon}(y,s)\Big|_{B_r(x,t)}^{} \right| \le \bar C \left( 1 + \frac{1}{\min_{(y,s) \in B_r(x,t)} (\textrm{dist}((y,s), \partial_p \Omega_T))^2}\right)
$$
in points $(y,s)$ where the gradient exists since the right hand side is independent of $\delta$.  Thus, due to the Lebesgue-Besicovitch Differentiation Theorem in \cite{16}, we have for almost every point $(x,t) \in \Omega_T$,
\begin{eqnarray*}
	|\nabla u_{\varepsilon}(x,t)| &\le& \lim_{r \to 0} \intav{B_r(x,t)} |\nabla u_{\varepsilon}(y,s)| d(y,s) \\
	&\le& \bar C \lim_{r \to 0}  \left( 1 + \frac{1}{\min_{(y,s) \in B_r(x,t)} (\textrm{dist}((y,s), \partial_p \Omega_T))^2}\right)\\
	&=& \bar C  \left(1+\frac{1}{\textrm{dist}((x,t), \partial_p \Omega_T)^2}\right).
\end{eqnarray*}

\end{proof}

As an immediate consequence we have the following result.

\begin{corollary}[Local Lipschitz regularity]\label{c4.1}
Let $\{u_{\varepsilon}\}_{\varepsilon >0}$ be a family of  solutions of \eqref{Equation Pe}. Let $ K \subset\Omega_T$ be a compact set and $\tau >0$ be such that $\mathcal{N}^{-}_{\tau}(K) \subset\Omega_T$. If $\eqref{E1}$-$\eqref{E2}$ hold, then there exists a constant $L=L(\tau,\lambda,\Lambda,c_0,c_1, \|\nabla f\|_{\infty}, \|\varphi\|_{\infty}, K)>0$ such that
$$
|\nabla u_{\varepsilon}(x,t)| \le L, \quad \forall(x,t) \in K.
$$
\end{corollary}
\begin{proof}
For $(x_0,t_0) \in K$, consider the function
$$
    w_{\varepsilon,r}(x,t) := \frac{1}{r} u_{\varepsilon}(x_0 + rx, t_0+r^2 t).
$$
For $r\in(0,\tau)$ we have that $w_{\varepsilon,r}$ is a solution of
\begin{eqnarray*}
\mathcal{L} w_{\varepsilon,r}- \partial_tw_{\varepsilon,r} &=& \zeta_{\varepsilon / r} (w_{\varepsilon,r}) + rf_{\varepsilon}\\
&\colon=&g_\varepsilon(x,t)
\end{eqnarray*}
in $B_{1} \times (-1,0)$. The result now follows from Corollary \ref{p4.1}.
\end{proof}

\subsection{Uniform regularity in time.}
Next, as was mentioned above, using the uniform Lipschitz continuity in the space variables, we obtain the uniform H\"{o}lder continuity in time. First, we need the following lemma.
\begin{lemma}\label{l4.1}
Let $0 \le u \in C(\overline{B}_{1}(0) \times [0, 1/(4n+M_0)])$ be such that
$$
|\mathcal{L}u-\partial_tu| \leq M_{0}\quad \textrm{in} \quad \{u>1\},
$$
for some $M_0 >0$, and $|\nabla u| \leq L$, for some $L>0$. Then there exists a constant $C=C(L)$ such that
$$
|u(0,t)-u(0,0)| \leq C, \quad \textrm{if} \quad 0 \leq t \leq \frac{1}{4n + M_0}.
$$
\end{lemma}
\begin{proof}
This lemma is a slight modification of  \cite[Lemma 4.1]{RMR}. Without loss of generality we may assume that $L>1$. We denote
$$
	c(p) \colon= \frac{np \Lambda}{n+p-2}, 
$$ 
where $\Lambda \colon = \max \left\{\frac{1}{p}, \frac{p-1}{p}\right\}$ denotes the greatest eigenvalue. We divide the proof into two steps. \\
{\bf Step 1.} \, First we claim that, if
$$
Q_{t_{0},t_{1}} := B_{1}(0) \times (t_0,t_1) \subset  \{u> 1\} \quad \textrm{for} \quad  t_1-t_0 \leq \frac{1}{4n+M_0},
$$
then
$$
    |u(0,t_1)-u(0,t_0)| \leq 2L.
$$
In fact, let
$$
    h^{\pm}(x,t):= u(0,t_0) \pm L \pm \frac{2L}{\Lambda} c(p)|x|^2 \pm (4nL + M_0)(t-t_0).
$$
Thus, Then for the specific $h^{\pm}$ we obtain
\begin{eqnarray*}
\partial_th^{+} - \mathcal{L}  h^+ &=&  (4nL + M_0) - \left(\frac{n}{p} + \frac{p-2}{p}\right) \frac{4L}{\Lambda}c(p) \\
&=& M_0\\
 \partial_th^{-} - \mathcal{L} h^{-} &=& -(4nL + M_0) + \left(\frac{n}{p} + \frac{p-2}{p}\right)\frac{4L}{\Lambda}c(p) \\
 &=&- M_0
\end{eqnarray*}

Set
$$
	t_2 := \sup_{t_0 \le \bar{t} \le t_1} \{\bar{t} : |u(0,t)-u(0,t_0)| \le 2L, \,\, \forall \, t_0 \le t \le \bar{t}\}.
$$
So  $t_{0} < t_{2} \leq t_{1}$ is such that
$$
    |u(0,t)-u(0,t_0)| \leq 2L, \quad \textrm{for} \quad t \in [t_0,t_2).
$$
Moreover, from the Lipschitz continuity in space, one has
$$
    h^{-} \leq u \leq h^{+} \quad \textrm{on}\quad\partial_{p}Q_{t_0,t_2}.
$$
On the other hand,
\begin{eqnarray*}
    \partial_th^{-} - \mathcal{L} h^{-}&\leq& -M_0 \leq \partial_tu - \Delta^N_p u\\
     &\leq& M_{0} \leq \partial_th^{+} - \mathcal{L} h^{+}.
\end{eqnarray*}
Therefore, by comparison principle (see Proposition \ref{Prop1})
$$
    h^{-} \leq u \leq h^{+} \quad \textrm{in}\quad Q_{t_0,t_2}.
$$
In particular, since $t_2-t_0 \leq t_1-t_0 \leq \frac{1}{4n+M_0}$ and $L>1$ one has
$$
    |u(0,t_2)-u(0,t_0)| < 2L.
$$
Because of the strict inequality above, we may take $t_2=t_1$ and therefore the claim is proved.\\
{\bf Step 2.}\, Let us consider now the cylinder $Q_{0,t}$ with $0 < t \leq \frac{1}{4n + M_0}$.\\
If $Q_{0,t} \subset \{u>1\}$, we apply Step 1 to get
$$
    |u(0,t)-u(0,0)| \leq 2L.
$$
If $Q_{0,t} \nsubseteq  \{u>1\}$, let $0 \leq t_{1} \leq t_{2} \leq t$ and $x_{1},x_{2} \in \overline{B}_{1}(0)$ be such that
$$
    0 \leq u(x_1,t_1) \leq 1, \,\,\,\, 0 \leq u(x_2,t_2) \leq 1
$$
and
$$
    (\overline{B}_{1}(0) \times (0,t_{1})) \cup (\overline{B}_{1}(0) \times (t_{2},t)) \subset  \{u>1\}.
$$
Then, Step 1 and the Lipschitz continuity in space provide
\begin{eqnarray*}
    |u(0,t)-u(0,0)| &\leq& |u(0,t)-u(0,t_2)| + |u(0,t_2) - u(x_2,t_2)| + |u(x_2,t_2)|\\
    &+& |u(x_1,t_1)| + |u(x_1,t_1)-u(0,t_1)| + |u(0,t_1)-u(0,0)|\\
    & \leq& 2(2L+L+1),
\end{eqnarray*}
which completes the proof.
\end{proof}
We are now ready to prove uniform H\"{o}lder continuity of solutions in time.


\begin{proposition}\label{p4.2}
Let $\{u_{\varepsilon}\}_{\varepsilon >0}$ be a family of solutions of $\eqref{Equation Pe}$. Let $K \subset\Omega_T$ be compact and $\tau >0$ be such that $\mathcal{N}_{2 \tau}(K) \subset \Omega_T$. If $\eqref{E1}$-$\eqref{E2}$ hold, then there exists a constant $C>0$, independent of $\varepsilon$, such that
$$
|u_{\varepsilon}(x,t+\Delta t) -u_{\varepsilon}(x,t)| \leq C |\Delta t|^{1/2},\,\,\,\textrm{ for }(x,t),(x,t+\Delta t) \in K.
$$
\end{proposition}
\begin{proof}
Let $r\in(0,\tau)$, $\varepsilon \le r$, $(x_0,t_0) \in K$ and $w_{\varepsilon,r}(x,t)$, $g_\varepsilon(x,t)$ be as in the proof of Corollary \ref{c4.1}.
From $\eqref{E1}$ and $\eqref{E2}$ we get, in the set $\{ w_{\varepsilon,r} \ge \varepsilon /r \}$,
$$
0 \le g_{\varepsilon}(x,t) \le r f_{\varepsilon}(x) + \frac{r}{\varepsilon} \zeta \left(\frac{r}{\varepsilon} w_{\varepsilon, r}\right) \le  \tau c_1=:C_{\star}.
$$
Also $|\nabla w_{\varepsilon,r}(x,t)|\leq L$. Therefore, we may apply Lemma \ref{l4.1}, with $M_0 = C_{\star}$, to obtain
$$
|w_{\varepsilon,r}(0,t) - w_{\varepsilon,r}(0,0)| \leq C, \quad \textrm{for}\quad 0 \leq t \leq \frac{1}{4n + C_\star},
$$
or in other terms
$$
|u_\varepsilon(x_0,t_0+r^2t)-u_\varepsilon(x_0,t_0)|\leq Cr,\,\,\,\textrm{ for }0\leq t\leq\frac{1}{4n+C_\star}.
$$
In particular, for $r\in(0,\tau)$, one has
\begin{equation}\label{4.3}
\left|u_{\varepsilon} \left(x_0, t_0 + \frac{r^2}{4n + C_\star}\right) - u_{\varepsilon} (x_0,t_0)\right| \leq Cr.
\end{equation}
Now if $(x_0,t_0+\Delta t) \in K$ and $0 < \Delta t < \frac{\tau^2}{4n + C_\star}$, taking $r = \Delta t^{1/2} \sqrt{4n+C_\star}$ in \eqref{4.3} leads to
$$
|u_{\varepsilon}(x_0,t_0 + \Delta t) - u_{\varepsilon}(x_0,t_0)| \leq C\sqrt{4n+C_\star} \Delta t^{1/2}.
$$
On the other hand, if $\Delta t \geq \frac{\tau^2}{4n+C_\star}$, thus we get
$$
|u_{\varepsilon}(x_0,t_0 + \Delta t) - u_{\varepsilon}(x_0,t_0)| \leq 2 \Upsilon \leq \frac{2 \Upsilon}{\tau} \sqrt{4n+C_\star} \Delta t^{1/2},
$$
which completes the proof.
\end{proof}

We are interested in geometric proprieties of the limiting function
$$
	u \colon= \lim_{k \to \infty} u_{\varepsilon_k},
$$
for a subsequence $\varepsilon_{k} \to 0$. From Theorem \ref{tel4.1} the family $\{u_{\varepsilon}\}$ is pre-compact in $\textrm{Lip}(1,1/2)(\Omega_T)$. Hence, up to a subsequence, there exists a limiting function $u$, obtained as the uniform limit of $u_{\varepsilon}$, as $\varepsilon \to 0$. One readily verifies that the  limiting function $u$ satisfies (see Theorem 5.1 in \cite{RMR} )
\begin{enumerate}
\item[(1)] $u$ is a solution of 
\begin{equation} \label{Eq1}
 \mathcal{L} u-u_t=f \cdot \chi_{\{u>0\}},
\end{equation}
 where $f$ is the uniform limit of $f_\varepsilon$, with $f$ satisfing \eqref{E1}-\eqref{E2};
\item[(2)] the function $t \mapsto u(x,t)$ is non-decreasing in time.
\end{enumerate}

\section{Scaling barriers and geometric nondegeneracy}\label{s6}

In this section we establish the exact growth of the solution near the free boundary, from which we deduce the porosity of its time level sets (see Theorem \ref{t6.1}).  The proof is quite similar to those in \cite{RMR}, but for the reader's convenience we decided to give the details.

\begin{lemma}\label{l6.1}
Let $u \in C(Q_1)$ be a viscosity solution to
$$
	\mathcal{L} u - u_t = f \quad \textrm{in} \quad U^+ \colon= \{u>0\}.
$$ 
 Then for every $(z,s) \in \overline{\{u>0\}}$ and $r >0$ with $Q_r(z,s) \subset Q_1$ we have
$$
\sup_{(x,t) \in \partial_p Q^{-}_{r}(z,s)} u(x,t) \ge \mu_0 r^{2} + u(z,s),
$$
where $\mu_0 = \min \left(\frac{pc_0}{4(n+p-2)}, \frac{c_0}{2}\right)$.
\end{lemma}
\begin{proof}
 Suppose that $(z,s) \in\{u>0\}$, and, for small $\delta>0$, set
\begin{eqnarray*}
\omega^{\delta}(x,t) &\colon=& u(x,t) - (1-\delta)u(z,s) \\
\psi(x,t)&\colon=& \left(\frac{pc_0}{4(n+p-2)}\right)|x-z|^{2} - \left(\frac{c_0}{2}\right)(t-s).
\end{eqnarray*}
Since $D_{ij} \psi = \frac{pc_0}{2(n+p-2)}\delta_{ij}$ then
\begin{eqnarray*}
\mathcal{L} \psi - \partial_t \psi &=& \frac{1}{p} \Delta \psi + \frac{p-2}{p} \left\langle  D^2 \psi \cdot \frac{\nabla \psi}{|\nabla \psi|}, \frac{\nabla \psi}{|\nabla \psi|}\right\rangle\\
&=&\frac{1}{p} \cdot \frac{n p c_0}{2(n+p-2)} +\frac{p-2}{p} \cdot  \frac{ pc_0}{2(n+p-2)} + \frac{c_0}{2} \\
&=&  \frac{n c_0}{2(n+p-2)} + \frac{(p-2) c_0}{2(n+p-2)} + \frac{c_0}{2}\\
&=&  \frac{c_0}{2} + \frac{c_0}{2}=c_0\\
	&\le&  f(x,t)  =  \Delta^N_p u - \partial_t u\\
	&=&\mathcal{L} \omega^{\delta}-\partial_t \omega^{\delta}.
\end{eqnarray*}
Moreover, $\omega^{\delta} \le \psi$ on $\partial\{u>0\}\cap Q^{-}_{r}(z,s)$. Note that we can not have
$$
\omega^{\delta} \le \psi \quad \textrm{on} \quad \partial_{p}Q^{-}_{r}(z,s)\cap\{u>0\},
$$
because otherwise we could apply the comparison principle(see Proposition \ref{Prop1})  to obtain
$$
\omega^{\delta} \le \psi \quad \textrm{in} \quad Q^{-}_{r}(z,s)\cap\{u>0\},
$$
which contradicts the fact that $\omega^{\delta}(z,s) = \delta u(z,s) >0 = \psi(z,s)$. Hence, for $(y,\tau) \in \partial_{p}Q^{-}_{r}(z,s)$ we must have
$$
\omega^{\delta}(y, \tau) > \psi(y,\tau) = \mu_0 r^{2}.
$$
Letting $\delta \to 0$ in the last inequality we conclude the proof.
\end{proof}


\subsection{A class of functions in the unit cylinder}

Next, we establish the growth rate of the solution near the free boundary, which is known for $p$-parabolic variational problems (see \cite{Sha}) but is new in the fully nonlinear framework. We start by introducing a class of functions.
\begin{definition}
We say that a function $u \in C(Q_1)$ is in the class $\Theta$ if $0\le u\le1$ in $Q_1$, 
$0 \le \mathcal{L} u-\partial_t u \le c_1$ in $Q_1$ and $\partial_t u \ge 0$,
in the viscosity sense, with $u(0,0)=0$.
\end{definition}

The following Theorem gives the growth of the elements of the family $\Theta$. This completes a result proved in \cite{RMR} for the fully nonlinear parabolic equations case. 
\begin{theorem} \label{l6.3}
If $u\in \Theta$, then there is a constant $C_0 = C_0(n,p,L,c_1)>0$ such that
$$
	|u(x,t)| \le C_0 (d(x,t))^2, \quad \forall \,\, (x,t) \in Q_{1/2},
$$
where
$$
d(x,t):=
\begin{cases}
  \sup \left\{ r : Q_r(x,t)\subset\{u>0\}\right\}, & \mbox{if } (x,t)\in\{u>0\} \\
  0, & \mbox{otherwise}.
\end{cases}
$$
\end{theorem}

To prove Lemma \ref{l6.3}, we need to introduce some notation. Set
$$
S(r,u,z,s):=\displaystyle \sup_{Q^{-}_{r}(z,s)}u.
$$
For $u\in \Theta$, we define
$$
H(u,z,s):= \left\{ j \in \mathbb{N} \cup \{0\} : S(2^{-j},u,z,s) \le M  S(2^{-j-1},u,z,s)\right\},
$$
where $M:=4\max(1,\frac{1}{\mu_0})$, with $\mu_0$ as in Lemma \ref{l6.1}. When $(z,s)$ is the origin, we suppress the point dependence. As in \cite{RMR}, we first show a weaker version of the inequality. 

\begin{lemma}\label{l6.2}
If $u\in \Theta$, then there is a constant $C_1 = C_1(n,c_1)>0$ such that
$$
S(2^{-j-1},u) \le C_12^{-2j},\,\,\,\,\forall j\in H(u).
$$
\end{lemma}
\begin{proof}
First, note that $H(u)\not=\emptyset$ because $0\in H(u)$. Indeed, using Lemma \ref{l6.1}, we have
$$
S(1,u)\leq1=4\left(\frac{1}{\mu_0}\right)\mu_02^{-2}\leq4\left(\frac{1}{\mu_0}\right)S(2^{-1},u)\leq MS(2^{-1},u).
$$
Next, suppose the conclusion of the lemma fails. Then, for every $k \in \mathbb{N}$, there is $u_k \in \Theta$ and $j_k \in H(u_k)$ such that
$$
S(2^{-j_k-1},u_k) \ge k 2^{-2j_k}.
$$
Define $v_k:Q_1 \rightarrow \mathbb{R}$ by
$$
v_k(x,t):= \frac{u_k(2^{-j_k}x, 2^{-2j_k}t)}{S(2^{-j_k-1},u_k)}.
$$
One easily verifies that
$$0 \le v_k \le 1 \quad \textrm{in} \ Q^{-}_1; \qquad 0 \le \mathcal{L} v_k-\partial_t v_k \le \frac{c_1}{k}; $$
$$\sup_{Q^{-}_{1/2}}v_k =1;\qquad v_k(0,0)=0; \qquad\partial_t v_k \ge 0 \quad \textrm{in} \ Q_1^-.$$
Therefore, from \cite[Theorem 1.1]{AM} we deduce that $v_k \in C^{1+\alpha, \frac{1+\alpha}{2}}_{\textrm{loc}}(Q^{-}_1)$ is uniformly bounded,  independently  of $k$, for a constant $\alpha \in (0,1)$. It follows then from Arzel\'{a}-Ascoli theorem that there exists a subsequence, still denoted $v_k$, and a function $v \in C^{1+\alpha,\frac{ 1+\alpha}{2}}(\overline{Q^{-}_{3/4}}) $ such that $v_k \to v$ and $\nabla v_k \to \nabla v$ uniformly in $\overline{Q^{-}_{3/4}}$. Moreover, 
\begin{equation} \label{cont}
	\sup_{Q^{-}_{1/2}} v =1
\end{equation}
and using compactness arguments (see \cite[Theorem 2.10]{Silv})
$$
\partial_t v-\mathcal{L} v=0 \quad \textrm{in} \quad Q^{-}_{3/4} , \quad v(0,0)=0, \quad v \ge 0, \quad \partial_t v \ge 0
$$
in $Q^{-}_{3/4}$.  Therefore, by the maximum principle (see \cite{Arg}), we obtain $v \equiv 0$. This gives us a contradiction to \eqref{cont}, if we choose $k \gg 1$.
\end{proof}

\begin{proof}[Proof of Theorem \ref{l6.3}]
Using Lemma \ref{l6.2},  it follows exactly as in \cite[Lemma 6.3]{RMR} that
\begin{equation}\label{6.1}
S(2^{-j},u)\leq 4C_12^{-2j},\,\,\,\forall j\in\mathbb{N}.
\end{equation}
Now, fix $r \in (0,1)$, by choosing $j \in \mathbb{N}$ such that $2^{-j-1} \le r \le 2^{-j} $, one has 
$$
	\sup_{Q^{-}_{r}(0,0)} \, u \le \sup_{Q^{-}_{2^{-j}}} \, u \le 4C_1 2^{-2j} = 16C_12^{-2j-2} \le 16C_1 r^2,
$$
i.e.,
\begin{equation} \label{6.2}
	S(r,u) \le 16C_1 r^2.
\end{equation}

To obtain a similar estimate for $u$ over the whole cylinder (and not only over its lower half) we use a barrier from above. Set
$$
\omega(x,t):= A_1|x|^2 + A_2t,
$$
where $A_2 = \frac{2(n+p-2)}{p} A_1$ and $A_1 >0$. Then in $Q^{+}_{1} = B_{1}(0) \times (0,1)$ one gets  that
\begin{eqnarray*}
	\mathcal{L}\omega - \partial_t \omega &=& \frac{2A_1n}{p} + 2A_1\frac{p-2}{p} -A_2\\
	& =&2\left(\frac{n}{p} + \frac{p-2}{p}\right)A_1 - A_2 = 0 \\
	&\le& \mathcal{L} u - \partial_t u.
\end{eqnarray*}
If $A_1$ is large enough, then $\omega \ge u$ on $\partial_p Q^{+}_{1}$, where for the estimate on $\{t=0\}$ we used $S(r,u)\le 16C_1r^2$ from \eqref{6.2}. Hence, by the comparison principle (see \cite[Theorem 2.9]{Silv} ) one has $\omega \ge u$ in $Q^{+}_{1}$. Therefore
$$
\sup_{Q_r(0,0)} u \le C_0 r^{2},
$$
for a constant $C_0>0$.
\end{proof}

\begin{proof}[Proof of Theorem \ref{t6.1}]
Using Theorem \ref{l6.3}, the proof follows exactly as the one of \cite{RMR}.
\end{proof}

\end{document}